\documentclass[11pt]{article}
\usepackage{amsmath,amstext,amssymb,amsthm, amsfonts}
\include{amssymb}
\include{amsmath,amsthm}
\include{amsfonts}
\include{mathrsfs}

\newcommand{\be}{\begin{equation}}
\newcommand{\ee}{\end{equation}}
\newcommand{\ba}{\begin{eqnarray}}
\newcommand{\ea}{\end{eqnarray}}
\newcommand{\bas}{\begin{eqnarray*}}
\newcommand{\eas}{\end{eqnarray*}}
\newtheorem{Example}{Example}
\newtheorem{Proposition}{Proposition}
\newtheorem{Lemma}{Lemma}

\def\E{\mathbb{E}}
\def\A{\mathbb{A}}
\def\H{\mathbb{H}}
\def\X{\mathbb{X}}
\def\X{\mathbb{X}}

\title{\LARGE \bf
Examples  concerning Abel and Ces\`aro limits }

\date{}

\begin{document}
\maketitle

\begin{center}
Christopher J. Bishop \footnote{Department of  Mathematics,
  Stony Brook University,
Stony Brook, NY 11794-3651, USA, bishop@math.sunysb.edu},
 Eugene~A.~Feinberg \footnote{Department of Applied Mathematics and
Statistics,
 Stony Brook University,
Stony Brook, NY 11794-3600, USA, eugene.feinberg@sunysb.edu}, \
and Junyu Zhang \footnote{School of Mathematics and Computational
Science,
  Sun Yat-sen University, Guangzhou, 510275, P. R. China,  mcszhjy@mail.sysu.edu.cn} \\

\smallskip
%July 3, 2007
\end{center}

\begin{abstract}
This note describes examples of all possible equality and strict
inequality relations between upper and lower Abel and Ces\`aro
limits  of sequences bounded above or below.  It also provides
applications to Markov Decision Processes.

% \cite{feinberg} and \cite{Schal} both
%presents sufficient conditions for the existence of stationary
%optimal policies for average cost Markov decision processes with
%Borel state and action spaces. From the view of mathematical
%analysis, the sufficient conditions in \cite{Schal} is stronger than
%the sufficient conditions in \cite{feinberg}. we would like to find
%an example which satisfies the sufficient conditions in
%\cite{feinberg}, but does not satisfy the sufficient conditions in
%\cite{Schal}. From the experience, it is impossible. But we can not
%prove it by now.

\end{abstract}

{\bf Keywords:} Tauberian theorem, Hardy-Littlewood theorem, Abel
limit, Ces\`aro limit

\section{Introduction}
For a sequence $\{u_n\}_{n=0,1,\ldots}$ consider  lower and upper
Ces\`aro limits
\[  \b{\sl C}=\liminf_{n\to\infty} \frac{1}{n}\sum_{i=0}^{n-1} u_i,\qquad
 \bar{C}=\limsup_{n\to\infty} \frac{1}{n}\sum_{i=0}^{n-1} u_i
\]
and lower and upper Abel limits
\[ \b{\it A}=\liminf_{\alpha\to 1-} (1-\alpha)\sum_{n=0}^\infty u_n\alpha^n,\qquad
 \bar{A}=\limsup_{\alpha\to 1-} (1-\alpha)\sum_{n=0}^\infty
 u_n\alpha^n.
\]

If a sequence $\{u_n\}_{n=0,1,\ldots}$ is bounded above or below
then, according to a Tauberian theorem (see, e.g.,
Sennott~\cite[pp. 281, 282]{Sennott1999}),
\begin{equation}\label{e:TaubT}
\b{\sl C}\le\b{\it A}\le \bar{A}\le \bar{C},
\end{equation}
and, according to the Hardy-Littlewood theorem (see, e.g.,
Titchmarsh~\cite[p. 226]{Titchmarsh}), if $\b{\it A}= \bar{A}$
then
\begin{equation}\label{e:HLT}
\b{\sl C}=\b{\it A}= \bar{A}= \bar{C}.
\end{equation}
In view of the Tauberian and Hardy-Littlewood theorems (1) and
(2), either equalities (\ref{e:HLT}) hold or only the following
relations can be possible:
\begin{equation}\label{e:3}
\b{\sl C}<\b{\it A}< \bar{A}< \bar{C},
\end{equation}
\begin{equation}\label{e:4}
\b{\sl C}=\b{\it A}<\bar{A}= \bar{C},
\end{equation}
\begin{equation}\label{e:5}
\b{\sl C}<\b{\it A}< \bar{A}= \bar{C},
\end{equation}
\begin{equation}\label{e:6}
\b{\sl C}=\b{\it A}< \bar{A}< \bar{C}.
\end{equation}

Hardy \cite{Hardy},  Liggett and Lippman~\cite{Liggett}, Sznajder
and Filar \cite[Example 2.2]{Sznajder},  Sennott~\cite[p.
286]{Sennott1999}, Keating and Reade~\cite{Keating}, and
Duren~\cite[Chapter 7]{Duren} provided at different levels of
details examples
of bounded sequences for which inequalities (\ref{e:3}) hold. %This
%example was also mentioned in Sznajder and Filar \cite[Example
%2.2]{Sznajder}, and the limits $\b{\sl C},$ $\b{\it A},$
%$\bar{A},$ and $\bar{C}$ for that sequence were computed in
%Keating and Reade \cite{Keating}. Another example, when
%inequalities (\ref{e:3}) hold, was provided in Liggett and
%Lippman~\cite{Liggett}.
%
This note demonstrates that inequalities (\ref{e:4})--(\ref{e:6})
may also take place for bounded sequences.  Example~\ref{Ex1}
demonstrates the possibility of (\ref{e:4}), and Example~\ref{Ex2}
demonstrates the possibility of (\ref{e:5}).  Of course,
inequalities (\ref{e:6}) hold for the sequence
$\{-u_n\}_{n=0,1,\ldots},$ if inequalities (\ref{e:5}) hold for a
sequence $\{u_n\}_{n=0,1,\ldots}.$

The Tauberian and Hardy-Littlewood theorems are important for many
applications. For example, %These theorems play important roles in
  they are
used to approximate average costs per unit time by total
discounted costs for Markov Decision Processes (MDPs) and
stochastic games; see e.g.,
\cite{feinberg,Hernandez-Lerma,Liggett,Schal,Sennott1999}.  They
are also used to evaluate long-run behavior of stochastic systems
by using Laplace-Stieltjes transforms, see e.g.,
Abramov~\cite{Abramov}. This study was motivated by applications
to MDPs; see
 Section~\ref{sec: MDP}.

%This
%study was motivated by applications %of (\ref{e:TaubT}) and
%%(\ref{e:HLT})
% to MDPs. By using (\ref{e:TaubT}),
%Sch\"al~\cite{Schal} proved the existence of stationary policies
%optimizing
% average costs per unit time for MDPs with compact action
%sets. This was done for problems with weakly and setwise
%continuous transition probabilities under a natural assumption on
%the behavior of the value functions for expected discounted
%rewards when discount factor is close to 1; see Assumption (B) in
%Feinberg et al.~\cite{feinberg}. These results were extended to
%general actions sets in Hern\'andez-Lerma~\cite{Hernandez-Lerma}
%for the case of setwise continuous transition probabilities and in
%Feinberg et al.~\cite{feinberg} for the case of weakly continuous
%transition probabilities. In addition, Feinberg et
%al.~\cite{feinberg} proved the existence of stationary optimal
%policies under  weaker Assumption (${\rm \underline B}$).
%According to \cite[Theorems 1,2,4]{feinberg}, though each of the
%Assumptions (B) and (${\rm \underline B}$) implies the existence
%of stationary optimal policies, under Assumption (B) stronger
%results hold than under Assumption (${\rm \underline B}$).  If
%inequalities (\ref{e:5}) were impossible, this would strengthen
%some of the results in Feinberg et al.~\cite{feinberg} under
%Assumption (${\rm \underline B}$). Thus it is useful to know that
%all inequalities (\ref{e:4})--(\ref{e:6}) are possible, which is
%the main conclusion of this note.

\section{Auxiliary facts}

\begin{Lemma}\label{l1} Let $\{L(n)\}_{n=0,1,\ldots}$ and $\{M(n)\}_{n=0,1,\ldots}$ be two sequences of nonnegative numbers,
$f_n(\alpha)=\alpha^{L(n)}-\alpha^{M(n)},$ $n=0,1,\ldots,$ and
$f^*(\alpha)=\sum_{n=0}^\infty f_n(\alpha)$. If
$L(n)\to\infty$ and $L(n)/M(n)\to 0$ as $n\to\infty,$ then:

(i) there is a sequence $\alpha_n\to 1-$ such that
$f_n(\alpha_n)\to 1$ as $n\to\infty$;

(ii) $\limsup_{\alpha\to 1-} f^*(\alpha)\ge 1.$
\end{Lemma}
\begin{proof} Since
$\lim_{\alpha\to 1-}f_n(\alpha)=0$ for all $n,$
$\limsup_{\alpha\to 1-} f^*(\alpha)=\limsup_{\alpha\to
1-}\sum_{n=m}^\infty f_n(\alpha)$ for any natural $m.$ Choose $m$
such that $L(n)<M(n)$  and $L(n)>1$ when $n\ge m.$

Observe that (i) implies (ii).  So, in the rest of the proof we
prove (i).

%
%If there is a sequence $\alpha_n\to 1-$ such that
%$f_n(\alpha_n)\to 1$ as $n\to\infty,$ then $\limsup_{\alpha\to 1-}
%f(\alpha)\ge \lim_{n\to\infty} f_n(\alpha_n)=1.$  In the rest of
%the proof, we construct such a sequence
%$\{\alpha_n\}_{n=1,2,\ldots} .$

By differentiating $f_n$ for each $n>m,$ observe that this
function reaches its maximum on $[0,1]$ at the point
\begin{equation}\label{e:21}\alpha_n=\left(\frac{L(n)}{M(n)}\right)^{\frac{1}{M(n)-L(n)}},\end{equation}
 and the maximum value is
\[f_n(\alpha_n)=\left(\frac{L(n)}{M(n)}\right)^{\frac{L(n)}{M(n)-L(n)}}-
\left(\frac{L(n)}{M(n)}\right)^{\frac{M(n)}{M(n)-L(n)}}.\] Since
$\frac{L(n)}{M(n)}\to 0$ as $n\to\infty,$ we have
$\frac{M(n)}{M(n)-L(n)}\to 1$ as $n\to\infty.$ Therefore,
\begin{equation}\label{e:22}\lim_{n\to\infty} f_n(\alpha_n)=\lim_{n\to\infty}\left(\frac{L(n)}{M(n)}\right)^{\frac{L(n)}{M(n)-L(n)}}=
\lim_{n\to\infty}\left(\frac{L(n)}{M(n)}\right)^{\frac{L(n)}{M(n)}\frac{M(n)}{M(n)-L(n)}}=1.\end{equation}
In addition, for $n>m$ \begin{equation*}1\ge
\left(\frac{L(n)}{M(n)}\right)^{\frac{1}{M(n)-L(n)}}\ge
\left(\frac{L(n)}{M(n)}\right)^{\frac{L(n)}{M(n)-L(n)}}\to 1\quad
{\rm as}\quad n\to\infty.\end{equation*} Thus, in view of
(\ref{e:21}) and  (\ref{e:22}),   $\alpha_n\to 1-$ and $
f_n(\alpha_n)\to 1$ as $n\to\infty.$
%\end{proof}

Recall that

\begin{equation}\label{e:8}\lim_{n\to
\infty}\frac{\sum_{k=1}^{n-1}k!}{n!}=0\qquad {\rm and}\qquad
\lim_{n\to \infty}\frac{\sum_{k=1}^{n}k!}{n!}=1.
\end{equation}
Indeed,
%We know that
%$$\sum_{n=1}^{\infty}\frac{n}{(n+1)!}=1,$$
%$$\sum_{k=1}^{n}kk!=(n+1)!-1.$$
\[ 0\le\lim_{n\to \infty}\frac{\sum_{k=1}^{n-1}k!}{n!}=
\lim_{n\to
\infty}\left[\frac{\sum_{k=1}^{n-2}k!}{n!}+\frac{(n-1)!}{n!}\right]
 \le  \lim_{n\to
\infty}\left\{\frac{(n-2)[(n-2)!]}{n!}+\frac{1}{n}\right\} =0.
%\lim_{n\to \infty}\left[\frac{(n-2)}{(n-1)n}+\frac{1}{n}\right]=0,
\]

\section{Examples}

For a sequence $\{u_n\}_{n=0,1,\ldots},$ define the function
\begin{equation}\label{e:Ab}f(\alpha)=(1-\alpha)\sum_{n=0}^\infty
u_n\alpha^n,\qquad \alpha\in [0,1).
\end{equation}

\begin{Example} \label{Ex1} %Inequalities (\ref{e:4}) hold with $\b{\sl C}=\b{\it A}=0$ and $\bar{C}=\bar{A}=1$.
{\rm For $D(k)=\sum_{i=1}^{k}i!,$  $k= 1,2,\ldots,$ let
%Taking  $n =2n_k-1$ we see the lower  Ces\`aro sum = 1/2 (again if the
%$n_k$ is large enough with respect to $n_{k-1}$)

%If $n_k=k!,$ $k\ge 0,$ then

 \begin{equation}\label{e:exx1}  u_{n}=\begin{cases}1, &{\rm if\ } D(2k-1)\le n< D(2k),\    k= 1,2,\ldots,\\
0, &{\rm otherwise}.\end{cases} \end{equation} }
\end{Example}

%Observe that $0\le \bar{u}_n<1$ for $n=1,2,\ldots\ .$ Therefore
%$\bar{u}_{n-1}<\bar{u}_n,$ if $ D(2k-1)\le n< D(2k)$ for $k=
%1,2,\ldots,$ and $\bar{u}_{n-1}>\bar{u}_n$ otherwise.  Thus
%$\{D(2k+1):\, k=0,1,\ldots\}$ and $\{D(2k):\, k=1,2\ldots\}$ are
%respectively the sets of local minima and maxima of the sequence
%$\{\bar{u}_n\}_{n=1,2,\ldots}.$  This implies
%%
%%
%%
%%Then, $u_{n},$ $n=0,1,2,3,\ldots,$ is as follows,
%%$$u_{n}: 0\quad 1 \quad 1 \quad \underbrace{0 \cdots 0}_{3-8}\quad \underbrace{1 \cdots 1}_{9-32}
%%\quad \underbrace{0 \cdots 0}_{33-152}
%% \quad\cdots\quad \underbrace{1 \cdots 1}_{\sum_{i=1}^{2k-1}i!\le n < \sum_{i=1}^{2k}i!}\quad
%%  \underbrace{0 \cdots 0}_{\sum_{i=1}^{2k}i!\le n < \sum_{i=1}^{2k+1}i!} \quad \cdots $$
%%
%\begin{equation}\label{e:7}\b{\sl C}=\liminf_{n\to \infty}\bar{u}_n=\liminf_{k\to
%\infty}\bar{u}_{D(2k+1)}=\lim_{k\to \infty}\frac{
%\sum_{i=1}^{k}(2i)!}{ \sum_{i=1}^{2k+1}i!} =0,\end{equation} where
%the last equality holds because of (\ref{e:8}),
%%We know that
%%$$\sum_{n=1}^{\infty}\frac{n}{(n+1)!}=1,$$
%%$$\sum_{k=1}^{n}kk!=(n+1)!-1.$$
%%
% and the third equality in (\ref{e:7}) holds because of %the explicit form of $D(2k+1)$ and
%$D(2i)-D(2i-1)=(2i)!.$
%
%Similar to (\ref{e:7}),
%\begin{equation}\label{e:9}\bar{C}=\limsup_{n\to \infty}\bar{u}_n=\limsup_{k\to
%\infty}\bar{u}_{D_{2k}}=\limsup_{k\to \infty}\frac{
%\sum_{i=1}^{k}(2i)!}{ \sum_{i=1}^{2k}i!}= 1-\lim_{k\to
%\infty}\frac{ \sum_{i=1}^{k}(2i-1)!}{
%\sum_{i=1}^{2k}i!}=1,\end{equation} where the last limit is 0 in
%view of (\ref{e:8}).
%

\begin{Proposition}
  Inequalities (\ref{e:4}) hold with $\b{\sl C}=\b{\it A}=0$ and $\bar{C}=\bar{A}=1$ for the sequence $\{u_n\}_{n=0,1,\ldots}$ defined in (\ref{e:exx1}).
\end{Proposition}
\begin{proof}
By using  properties of geometric series, observe that
%Denote $f(\alpha):=(1-\alpha)\sum_{n=0}^{\infty}u_{n}\alpha^{n}$.
\begin{equation}\label{e:10}f(\alpha)=\sum_{n=1}^{\infty}f_n(\alpha),\end{equation} where
\begin{equation}\label{e:11}f_n(\alpha)=\alpha^{D(2n-1)}-\alpha^{D(2n)}\ge
0.\end{equation}
%\label{fact2}
In view of (\ref{e:8}),   $D(2n-1)\to \infty$ and
$D(2n-1)/D(2n)\to 0$ as $n\to\infty.$ Formulas %(\ref{e:TaubT}),
(\ref{e:10}), (\ref{e:11}) and Lemma~\ref{l1}(ii) imply that
$1\ge\bar{A}=\limsup_{\alpha\to 1-}f(\alpha)\ge 1.$ Thus,
$\bar{A}=1.$  In view of (\ref{e:TaubT}), $\bar{A}\le \bar{C}$.
Since $\bar{C}\le 1$, then $\bar{C}=1.$

%\end{proof}

Now we show that $\b{\it A}=0.$
%
%This completes the analysis for the upper Abel limit in Example
%(\ref{Ex1}). Next we consider its lower Abel limit.
%
%\begin{Proposition} \label{Ex32}%$$\limsup_{\alpha \to
%%1^{-}}f_{3}(\alpha)=\limsup_{\alpha \to
%%1^{-}}(1-\alpha)\sum_{n=0}^{\infty}u_{n}\alpha^{n}=1-\alpha^{3}+\sum_{n=2}^{\infty}(\alpha^{2n!}-\alpha^{(n+1)!+1})=1.$$
%$$\liminf_{\alpha \to
%1^{-}}\sum_{n=1}^{\infty}\left(\alpha^{\sum_{i=1}^{2n-1}i!}
%-\alpha^{\sum_{i=1}^{2n}i!}\right) =0.$$
%\end{Proposition}
%
%
%
%\noindent{\bf Proof.}
%
%
For each $k= 1,2,\ldots,$ consider the sequence
$\{u^k_{n}\}_{n=0,1,\ldots}$
 \[u_n^{k}=\begin{cases}1, &{\rm if\ } n<D(2k)\ {\rm or}\  n\ge D(2k+1),\\
0, &{\rm otherwise}.\end{cases} \]

Let $f^k$ be the function $f$ from (\ref{e:Ab}) for the sequence
$\{u^k_{n}\}_{n=0,1,\ldots}$.  Since $u_n\le u_n^k$ for each
$n=0,1,\ldots,$ $f(\alpha)\le f^k(\alpha)$ for all $\alpha\in
[0,1),$ $k=1,2,\ldots \ .$  Therefore, to prove that $\b{\it
A}=0,$ it is sufficient to show the existence of a sequence
$\alpha_k\to 1-$ as $k\to\infty$ such that $\lim_{k\to\infty}
f^k(\alpha_k)=0.$

Observe that
\begin{equation*} f^k(\alpha)=(1-\alpha)
\left[\sum_{n=0}^{D(2k)-1} \alpha^n
 + \sum_{n=D(2k+1)}^\infty \alpha^n \right]\\
%& =&  (1-\alpha)\left[ \frac{1-\alpha^{\sum_{i=1}^{2k}i!}}{1-\alpha}
%+
%\frac{\alpha^{\sum_{i=1}^{2k+1}i!}}{1-\alpha}\right ]\\
=1-\alpha^{D(2k)}+\alpha^{D(2k+1)}.
\end{equation*}

In view of Lemma \ref{l1}(i),  there exist $\alpha_k\to 1-$ such
that $ \alpha_{k}^{D(2k)}-\alpha_{k}^{D(2k+1)}\to 1$ as
$k\to\infty.$
%
%By differentiating $f^{k}(\alpha),$ observe that this function
%achieves its minimum on $[0,1]$ at the point
%\[\alpha_k=\left (\frac{D(2k)}{D(2k+1)}\right)^{\frac{1}{(2k+1)!}}\in (0,1),\]
% and the minimum value is
%\[f^{k}(\alpha_{k})=1-\left(\frac{D(2k)}{D(2k+1)}\right)^{\frac{D(2k)}{(2k+1)!}}-
%\left(\frac{D(2k)}{D(2k+1)}\right)^{\frac{D(2k+1)}{(2k+1)!}}\] In
%view of (\ref{e:8}), $\frac{D(2k)}{D(2k+1)}\to 0$ and
%$\frac{D(2k+1)}{(2k+1)!}\to 1$ as $k\to\infty.$ Therefore,
%\[\lim_{k\to\infty} f^{k}(\alpha_{k})=1-\lim_{k\to\infty}\left(\frac{D(2k)}{D(2k+1)}\right)^{\frac{D(2k)}{(2k+1)!}}=
%1-\lim_{k\to\infty}\left(\frac{D(2k)}{D(2k+1)}\right)^{\frac{D(2k)}{D(2k+1)}\frac{D(2k+1)}{(2k+1)!}}=0.\]
%In addition, \[1\ge
%\left(\frac{D(2k)}{D(2k+1)}\right)^{\frac{1}{(2k+1)!}}\ge
%\left(\frac{D(2k)}{D(2k+1)}\right)^{\frac{D(2k)}{(2k+1)!}}\to
%1\quad {\rm as}\quad k\to\infty.\]
%
Thus, %$\alpha_k\to 1-$ and
$ f^k(\alpha_k)\to 0$ as $k\to\infty,$ and $\b{\it A}=0.$  This
implies $\b{\sl C}=0$ since $0\le \b{\sl C}\le \b{\it A}.$
\end{proof}

\begin{Example}\label{Ex2}  %Inequalities (\ref{e:5}) hold with $\b{\sl C}=\frac{1}{2},$ $\b{\it A}=\frac{3}{4},$ and
%$\bar{C}=\bar{A}=1$.
 {\rm Let}

%
% $$u_{n}=\left\{\begin{array}{ll}1,&n=0;\\0,& k!\le n<2k!, k\ge 1;\\1, & 2k!\le n< (k+1)!,k\ge 2.\end{array} \right.$$

\begin{equation}\label{e:exx2}
u_{n}=\begin{cases}0, &{\rm if\ } k!\le n<2k!,\ k= 1,2,\ldots,\\
1, &{\rm otherwise}.\end{cases}
\end{equation}

\end{Example}
\begin{Proposition}
Inequalities (\ref{e:5}) hold with $\b{\sl C}=\frac{1}{2},$ $\b{\it A}=\frac{3}{4},$ and
$\bar{C}=\bar{A}=1$  for the sequence $\{u_n\}_{n=0,1,\ldots}$ defined in (\ref{e:exx2}).
\end{Proposition}
\begin{proof}
By (\ref{e:8})
\[ \b{\sl C}=\lim_{n\to
\infty}\frac{1+\sum_{k=2}^{n-1}\left[(k+1)!-2k!\right]}{2n!}
=\lim_{n\to
\infty}\frac{\sum_{k=3}^{n}(k)!-\sum_{k=2}^{n-1}2k!}{2n!}=\frac{1}{2},
\] and
\[ \bar{C}
 =\lim_{n\to \infty}\frac{n!-\sum_{k=1}^{n-1}k!} {n!}
=1.\]

%Below we apply to the sequence  $\{u_n\}_{n=0,1,\ldots}$ from this example the notation $f(\alpha)$ introduced  in (\ref{e:Ab}).
By using the formula for the sum of geometric series,
%Denote $f(\alpha):=(1-\alpha)\sum_{n=0}^{\infty}u_{n}\alpha^{n}$.

\[f(\alpha)
=1-\alpha+\sum_{n=1}^{\infty}\left(\alpha^{2n!}-\alpha^{(n+1)!}\right).\]
%=\sum_{n=0}^{\infty}\left(\alpha^{2n!}-\alpha^{(n+1)!}\right).\]
By Lemma~\ref{l1}(ii), $\bar{A}\ge 1$.  However, $\bar{A}\le
\bar{C}=1.$  Thus, $\bar{A}= 1$.

%Using the same reasoning and idea as the proof of $\bar{A}=1$ in
%Example (\ref{Ex1}), we can prove \[\limsup_{\alpha \to
%1^{-}}\sum_{n=0}^{\infty}\left(\alpha^{2n!}-\alpha^{(n+1)!}\right)=1.\]
%\
%
%
%This completes the analysis for the upper Abel limit. Next we
%consider the lower Abel limit.

%Since $3\ln(\frac{1}{2}+\frac{1}{12})=$
%
%
%
%Suppose that $$\left (\frac{n!+1}{2n!}\right
%      )^{\frac{1}{n!-1}}<\left (\frac{(n+1)!+1}{2(n+1)!}\right
%      )^{\frac{1}{(n+1)!-1}}.$$
%
%      $$\left (\frac{n!+1}{2n!}\right
%      )^{(n+1)!-1}<\left (\frac{(n+1)!+1}{2(n+1)!}\right
%      )^{n!-1}$$
%      $$\left (\frac{1}{2}+\frac{1}{2n!}\right
%      )^{(n+1)!-1}<\left (\frac{1}{2}+\frac{1}{2(n+1)!}\right
%      )^{n!-1}$$
%
%       $$[(n+1)!-1]\ln{\left (\frac{1}{2}+\frac{1}{2n!}\right
%      )}<(n!-1)\ln{\left (\frac{1}{2}+\frac{1}{2(n+1)!}\right
%      )}$$
%
%      $$[(n+1)!-1]\ln{\left (\frac{1}{2}+\frac{1}{2n!}\right
%      )}<[(n+1)!-1]\ln{\left (\frac{1}{2}+\frac{1}{12}\right
%      )}<(n!-1)\ln{\left (\frac{1}{2}+\frac{1}{2(n+1)!}\right
%      )}$$
%
%      $$[(n+2)!-1]\ln{\left (\frac{1}{2}+\frac{1}{2(n+1)!}\right
%      )}<[(n+1)!-1]\ln{\left (\frac{1}{2}+\frac{1}{2(n+2)!}\right
%      )}$$

To compute $\b{\it A},$ define \[g(\alpha)=
1-f(\alpha)=\sum_{n=1}^{\infty}\left(\alpha^{n!}-\alpha^{2n!}\right)\]
and $\bar B=\limsup_{\alpha\to 1-} g(\alpha).$  Then $\b{\it
A}=1-\bar B.$

We compute $\bar B$ first.  Let
$g_n(\alpha)=\alpha^{n!}-\alpha^{2n!},$ $n=1,2,\ldots\ .$ When
$\alpha\in [0,1],$ the function $g_n(\alpha)$ reaches its maximum
at $\alpha_n=2^{-\frac{1}{n!}}$ and $g_n(\alpha_n)=\frac{1}{4}.$
In addition, this function increases on the interval
$[0,\alpha_n]$ and decreases on the interval $[\alpha_n,1].$

Let $\beta_k=2^{-\frac{1}{(k-1)!\sqrt{k}}},$ $k=1,2,\ldots\ .$
When $\alpha\in \left[\beta_k,\beta_{k+1}\right],$ $k=1,2,\ldots,$
then, if $n<k,$ the function $g_n(\alpha)$ decreases and reaches
its maximum on this interval at the point $\beta_k$; if $n>k$ then
it increases and reaches the maximum at the point $\beta_{k+1};$
and, if $n=k,$ it achieves the maximum at $\alpha_k.$ Thus,
\begin{equation}\label{e:14}g(\alpha)=\sum_{n=1}^{k-1}g_n(\alpha)+g_k(\alpha)
+\sum_{n=k+1}^{\infty}g_n(\alpha)<
\sum_{n=1}^{k-1}g_n(\beta_k)+g_k(\alpha)
+\sum_{n=k+1}^{\infty}g_n(\beta_{k+1}).\end{equation} Observe that
\begin{equation}\label{e:15}
\sum_{n=1}^{k-1}g_n(\beta_k)=\sum_{n=1}^{k-1}\beta_{k}^{n!}\left(1-\beta_{k}^{n!}\right)<\sum_{n=1}^{k-1}\left(1-\beta_{k}^{n!}\right)
<\frac{\sum_{n=1}^{k-1} n!}{(k-1)!\sqrt{k}}\ln{2}\to 0\quad{\rm
as}\quad k\to\infty,
\end{equation}
where the last inequality follows from $2^{-x}>1-x\ln{2}$ for
$x>0,$ and
\begin{equation}\label{e:16}
\sum_{n=k+1}^{\infty}g_n(\beta_{k+1})<
\sum_{n=k+1}^{\infty}\beta_{k+1}^{n!}=\sum_{n=k+1}^{\infty}2^{-\frac{n!}{k!\sqrt{k+1}}}
=2^{-\sqrt{k+1}}\sum_{n=k+1}^{\infty}2^{-\frac{n!-(k+1)!}{k!\sqrt{k+1}}}\le 2^{1-\sqrt{k+1}},
\end{equation}
where the last inequality holds because
$\frac{n!-(k+1)!}{k!\sqrt{k+1}}\ge n-(k+1)$ when $n\ge (k+1).$
Thus (\ref{e:15}) and (\ref{e:16}) imply that
\begin{equation}\label{e:17}
\lim_{k\to\infty} \left(\sum_{n=1}^{k-1}g_n(\beta_k)
+\sum_{n=k+1}^{\infty}g_n(\beta_{k+1})\right)=0.
\end{equation}
In conclusion,
\[\bar{B}=\limsup_{k\to\infty}\sup_{\alpha\in [\beta_k,\beta_{k+1}]}
g(\alpha)\le \lim_{k\to\infty}
\left(g_k(\alpha_k)+\sum_{n=1}^{k-1}g_n(\beta_k)
+\sum_{n=k+1}^{\infty}g_n(\beta_{k+1})\right)=\frac{1}{4},\] where
the first equality holds since $\beta_k\to 1,$  the inequality
holds  because of (\ref{e:14}) and because the function $g_k$
reaches its maximum at $\alpha_k$ on the interval $[0,1]$, and the
last equality holds because of $g_k(\alpha_k)=\frac{1}{4}$ and
(\ref{e:17}).  In addition, $\bar{B}\ge \lim_{k\to\infty}
g(\alpha_k)\ge \lim_{k\to\infty} g_k(\alpha_k)=\frac{1}{4}.$
 Thus $\bar{B}=\frac{1}{4}$ and
$\b{\it A}=1-\bar{B}=\frac{3}{4}.$ \end{proof}

\section{On approximations of average costs per unit time by normalized discounted costs for MDPs}  \label{sec: MDP}
Average costs for an MDP can be defined either as upper or as
lower limits of expected  costs per unit time over  finite time
horizons as the time horizon lengths tend to infinity. For each of
these two definitions of average costs, the minimal value is the
infimum of average costs taken over the set of all policies.  As
shown below, if the state space is infinite, Example~\ref{Ex2}
implies that it is possible that one of these two minimal values
can be approximated by normalized total expected discounted costs,
while such  approximations for another one are impossible.

Consider an MDP with a state space
$\X$, action space $\A$, sets of available actions $A(x)$,
transition probabilities $p$, and one-step cost $c$.  Here we
assume that:
\begin{itemize}
\item[(i)] the state space $\X$ is a nonempty countable set,
\item[(ii)] the  action space $\A$ is a measurable space
$(\A,{\cal A})$ such that all its singletons are measurable
subsets, that is, $\{a\}\in{\cal A}$ for each $a\in\A;$
\item[(iii)] for each state $x\in \X$ the set of available actions
$A(x)$ is nonempty and belongs to $\cal A;$ \item[(iv)] if an
action $a\in A(x)$ is chosen at a state $x\in\X$, then $p(y|x,a)$,
where $y\in\X$, is the probability that y is the state at the next
step; it is assumed that $p(\cdot|x,a)$ is a probability mass
function on $\X$ and $p(y|x,\cdot)$ is a measurable function on
$A(x);$ \item[(v)] if an action $a\in A(x)$ is selected at a state
$x\in\X$, then the one-step cost $c(x,a)$ is incurred; it is
assumed that the values $c(x,a)$ are uniformly bounded below, and
the function $c(x,\cdot)$ is measurable on $A(x)$ for each
$x\in\X.$
\end{itemize}

Let $\H_n=\X\times(\A\times \X)^n$ be the set of trajectories up
to the step $n=0,1,\ldots\ .$   For $n=1,2,\ldots,$ consider the
sigma-field ${\cal F}_n$ on $\H_n$ defined as the products of the
sigma-fields of all subsets of $\X$ and $\cal A.$  A policy $\pi$
is a sequence $\{\pi_n\}_{n=0,1,\ldots}$ of transition
probabilities from $\H_n$ to $\A$ such that: (i) for each
$h_n=x_0a_0x_1...a_nx_n\in \H_n,$ $n=0,1,\ldots,$ the probability
$\pi_n(\cdot|h_n)$ is defined on $(\A,{\cal A}),$ and it satisfies
the condition $\pi_n(A(x_n)|h_n)=1,$  and (ii) $\pi_n(B|\cdot)$ is
a measurable function on $(\H_n,{\cal F}_n)$ for each $B\in {\cal
A}.$  A policy $\pi$ is called stationary if there is a mapping
$\phi:\, \X\to \A$ such that $\phi(x)\in A(x)$ for all $x\in\X$
and $\pi_n(\{\phi(x_n)\}|x_0a_0x_1\ldots x_n)=1$ for all
$n=0,1,\ldots,$  $x_0a_0x_1\ldots x_n\in\H_n.$  Since a stationary
policy is defined by a mapping $\phi$, it is also denoted by
$\phi$ with a slight abuse of notations. Sometimes in the
literature, a stationary policy is called nonrandomized
stationary, deterministic stationary, or deterministic.  Let $\Pi$
be the set of all policies.

The standard arguments based on the Ionescu Tulcea
theorem~\cite[Chapter 5, Section 1]{Ne} imply that each initial
state $x$ and policy $\pi$ define a stochastic sequence on the
sets of trajectories $x_0a_0x_1a_1,...\ .$ We denote
by %$P_x^\pi$ the probabilities an by
$\E_x^\pi$ expectations
for this stochastic sequence.

For an initial state $x\in \X$ and for a policy $\pi,$  the
average cost per unit time is

\begin{equation*}\label{eq:avd}
w^*(x,\pi)=\limsup_{N\to\infty}\frac{1}{N}\E_x^\pi\sum_{n=0}^{N-1}c(x_n,a_n)=\limsup_{N\to\infty}\frac{1}{N}\sum_{n=0}^{N-1}\E_x^\pi
c(x_n,a_n). \end{equation*}

In general, if the performance of a policy $\pi$ is evaluated by a
function $g(x,\pi)$ with values in $[-\infty,\infty],$ where
$x\in\X$ is the initial state, we define the value function
$g(x)=\inf_{\pi\in\Pi} g(x,\pi).$ For $\epsilon\ge 0$, a policy
$\pi $ is called $\epsilon$-optimal, if $g(x,\pi)\le
g(x)+\epsilon$ for all $x\in \X.$  A 0-optimal policy is called
optimal.

For a constant $\alpha\in [0,1),$ called the discount factor,  the
expected total discounted costs are
\begin{equation*}\label{eq:disc}
v_\alpha(x,\pi)=\E_x^\pi\sum_{n=0}^\infty\alpha^n
c(x_n,a_n)=\sum_{n=0}^{\infty}\alpha^n\E_x^\pi c(x_n,a_n).
\end{equation*}

In general, proofs of the existence of stationary optimal policies
for expected average costs per unit time are more difficult than
for expected total discounted costs.  Average costs per unit time
are often analyzed by approximating $w^*(x,\pi)$ with
$(1-\alpha)v_\alpha(x,\pi)$ for the values of $\alpha$ close to 1.
Let
\begin{equation*}\label{eq:discul}
{\bar w}(x,\pi)=\limsup_{\alpha\to 1-}(1-\alpha)v_\alpha(x,\pi).
\end{equation*}
In addition to the upper limit of the average expected costs
(\ref{eq:avd}), consider the lower Ces\`aro limit
\begin{equation}\label{eq:avdl}
w_*(x,\pi)=\liminf_{N\to\infty}\frac{1}{N}\E_x^\pi\sum_{n=0}^{N-1}c(x_n,a_n)=\liminf_{N\to\infty}\frac{1}{N}\sum_{n=0}^{N-1}\E_x^\pi
c(x_n,a_n) \end{equation} and the lower Abel limit
\begin{equation*}\label{eq:discll}
\b{\it w}(x,\pi)=\liminf_{\alpha\to 1-}(1-\alpha)v_\alpha(x,\pi).
\end{equation*}

In view of the Tauberian theorem % (\ref{e:TaubT}), for each state $x\in\X$ and for each policy $\pi\in\Pi$
\begin{equation*}\label{eq:taubmdp} w_*(x,\pi)\le \b{\it w}(x,\pi)\le {\bar w}(x,\pi)\le
w^*(x,\pi), \qquad x\in\X,\ \pi\in\Pi.\end{equation*} Therefore,
the same inequalities hold for the values,
\[w_*(x)\le \b{\it w}(x)\le {\bar w}(x)\le
w^*(x), \qquad x\in\X.\] The natural questions are whether
$w^*(x)= {\bar w}(x)$ and whether $w_*(x)= \b{\it w}(x)$?

Let the state space $\X$ be finite.  Then, according to  Dynkin and
Yushkevich~\cite[Chapter 7, Section 3]{DY}, for each stationary policy $\phi$
\begin{equation}\label{eq:taubmdpst} w_*(x,\phi)= \b{\it w}(x,\phi)= {\bar
w}(x,\phi)= w^*(x,\phi), \qquad x\in\X.\end{equation} Though for
some $\epsilon > 0$ stationary $\epsilon$-optimal policies may not
exist for MDPs with finite state and arbitrary action sets (see
Dynkin and Yushkevich~\cite[Chapter 7, Section 8, Example 2]{DY}),
as proved in Feinberg~\cite[Corollary 1]{F80},
\begin{equation}\label{eq:taubmdpval} w_*(x)= \b{\it w}(x)= {\bar
w}(x)= w^*(x), \qquad x\in\X.\end{equation}

Equalities (\ref{eq:taubmdpst})  may not hold, when a stationary
policy $\pi$ is substituted with an arbitrary policy $\pi$. In
fact, all four situations presented in (\ref{e:3})--(\ref{e:6})
are possible with $\bar{C}=w^*(x,\pi),$ ${\bar A}={\bar
w}(x,\pi),$ $\b{\it A}=\b{\it w}(x,\pi),$ and $\b{\sl
C}=w_*(x,\pi).$  Indeed, consider an MDP with a single state and
two actions, that is, $\X=\{x\}$ and $\A=A(x)=\{a,b\}.$ Let also
$c(x,a)=1$ and $c(x,b)=0.$  In addition, $p(x|x,a)=p(x|x,b)=1$
since the process is always at state $x.$ Let at each step
$n=0,1,\ldots$ a policy $\pi$ select actions $a$ and $b$ with
probabilities $\pi_n(a)$ and  $\pi_n(b)$ respectively.  For a
sequence $\{u_n\}_{n=0,1,\ldots},$ let $\pi_n(a)=u_n.$  Then
$\E_x^\pi c(x_n,a_n)=u_n,$ $n=0,1,\ldots,$ and the values of
$w^*(x,\pi),$ ${\bar w}(x,\pi),$ $\b{\it w}(x,\pi)$, and
$w_*(x,\pi)$ are equal to the corresponding Ces\`aro and Abel limits
for the sequence $\{u_n\}_{n=0,1,\ldots}.$ Since all the
inequalities (\ref{e:3})--(\ref{e:6}) are possible for Ces\`aro and
Abel limits of bounded sequences  $\{u_n\}_{n=0,1,\ldots},$ these
inequalities are also possible for $\bar{C}=w^*(x,\pi),$ ${\bar
A}={\bar w}(x,\pi),$ $\b{\it A}=\b{\it w}(x,\pi),$ and $\b{\sl
C}=w_*(x,\pi).$ .

Now let $\X$ be countably infinite.  For each sequence $\{u_n\}_{n=0,1,\ldots}$ consider the MDP with the state space $\X=\{0,1,\ldots\},$ a single action $a$, that is $\A=\{a\},$ transition probabilities $p(x+1|x,a)=1,$ and one-step costs $c(x,a)=u_x,$ $x\in X.$  For this MDP, there is only one policy, and this policy is stationary. We denote this policy by $\phi$ and observe that $\E_0^\phi c(x_n,a_n)=u_n,$ $n=0,1,\ldots\ .$  Thus, equalities (\ref{eq:taubmdpst}) and (\ref{eq:taubmdpval}) may not hold.  In addition, all the inequalities (\ref{e:3})--(\ref{e:6}) are possible with ${\bar C}=w^*(x,\phi),$
${\bar A}={\bar w}(x,\phi),$  $\b{\it A}=\b{\it w}(x,\phi),$
$\b{\sl C}=w_*(x,\phi)$  and  with $\bar{C}=w^*(x),$
${\bar A}={\bar w}(x),$  $\b{\it A}=\b{\it w}(x),$
$\b{\sl C}=w_*(x).$  In particular, $w^*(x)={\bar w}(x)$ does not imply $w_*(x)=\b{\it w}(x),$ and $w_*(x)=\b{\it w}(x)$
does not imply $w^*(x)={\bar w}(x).$
\end{proof}

\section{Acknowledgements} Research of the first coauthor was partially
supported by NSF grant  DMS-1305233.  Research of the second
coauthor was partially supported by NSF grant  CMMI-1335296.
Research of the third coauthor was partially supported   by NSFC,
RFDP, CSC, the Fundamental Research Funds for the Central
Universities, and by Guangdong Province Key Laboratory of
Computational Science. This paper was written when Junyu Zhang was
visiting the Department of Applied Mathematics and Statistics,
Stony Brook University, and she thanks the department for its
hospitality.

\end{document}